\documentclass{amsart}
\usepackage{mathptmx}
\usepackage{amssymb}
\usepackage{amsfonts}
\usepackage{amsmath}
\usepackage{graphicx}
\usepackage{shadow}
\usepackage{color}
\usepackage[all]{xy}
\usepackage{hyperref}
\newtheorem{thm}{Theorem}[section]
\newtheorem{corollary}[thm]{Corollary}
\newtheorem{lemma}[thm]{Lemma}
\newtheorem{proposition}[thm]{Proposition}

\theoremstyle{definition}
\newtheorem{definition}{Definition}
\newtheorem{example}{Example}

\theoremstyle{remark}

\DeclareMathOperator{\Hom}{Hom}
\DeclareMathOperator{\Ker}{Ker}

\DeclareMathOperator{\wgdim}{w.gl.dim}
\DeclareMathOperator{\gdim}{gl.dim}
\DeclareMathOperator{\wdim}{w.dim}

\DeclareMathOperator{\pdim}{p.dim}

\DeclareMathOperator{\Ann}{Ann}
\DeclareMathOperator{\Nil}{Nil}
\DeclareMathOperator{\Ze}{Z}

\DeclareMathOperator{\m}{\frak{m}}
\DeclareMathOperator{\End}{End}

\newcommand{\field}[1]{\mathbb{#1}}

\newcommand{\Z}{\field{Z}}

\begin{document}

\title[Weak global dimension of Pr\"ufer-like rings]{Weak global dimension of Pr\"ufer-like rings}

\author{K. Adarbeh}
\address{Department of Mathematics and Statistics, KFUPM, Dhahran 31261, KSA}
\email{khalidwa@kfupm.edu.sa}

\author[S. Kabbaj]{S. Kabbaj $^{(1)}$}\thanks{$^{(1)}$ Corresponding author.}
\address{Department of Mathematics and Statistics, KFUPM, Dhahran 31261, KSA}
\email{kabbaj@kfupm.edu.sa}

\date{\today}

\subjclass[2010]{13F05, 13B05, 13C13, 16D40, 16B50, 16D90}

\keywords{Pr\"ufer domain, arithmetical ring, chained ring, fqp-ring, Gaussian ring, Pr\"ufer ring, semihereditary ring, quasi-projective module, weak global dimension, trivial ring extension.}

\begin{abstract}
In 1969, Osofsky proved that a chained ring (i.e., local arithmetical ring) with zero divisors has infinite weak global dimension; that is, the weak global dimension of an arithmetical ring is 0, 1, or $\infty$. In 2007, Bazzoni and Glaz studied the homological aspects of Pr\"ufer-like rings, with a focus on Gaussian rings. They proved that Osofsky's aforementioned result is valid in the context of coherent Gaussian rings (and, more generally, in coherent Pr\"ufer rings). They closed their paper with a conjecture sustaining that ``the weak global dimension of a Gaussian ring is 0, 1, or $\infty$." In 2010, the authors of \cite{BKM} provided an example of a Gaussian ring which is neither arithmetical nor coherent and has an infinite weak global dimension. In 2011, the authors of \cite{AJK} introduced and investigated the new class of fqp-rings which stands strictly between the two classes of arithmetical rings and Gaussian rings. Then, they proved the Bazzoni-Glaz conjecture for fqp-rings.

This paper surveys a few recent works in the literature on the weak global dimension of Pr\"ufer-like rings making this topic accessible and appealing to a broad audience. As a prelude to this, the first section of this paper provides full details for Osofsky's proof of the existence of a module with infinite projective dimension on a chained ring. Numerous examples -arising as trivial ring extensions- are provided to illustrate the concepts and results involved in this paper.
\end{abstract}
\maketitle


\begin{section}{Introduction}\label{I}

\noindent All rings considered in this paper are commutative with identity
element and all modules are unital. Let $R$ be a ring and $M$ an
$R$-module. The weak (or flat) dimension (resp., projective
dimension) of $M$, denoted $\wdim_{R}(M)$ (resp., $\pdim_{R}(M)$),
measures how far $M$ is from being a flat (resp., projective)
module. It is defined as follows: Let $n$ be an integer $\geq0$.
We have $\wdim_{R}(M)\leq n$ (resp., $\pdim_{R}(M)\leq n$) if there is a
flat (resp., projective) resolution $$0\rightarrow E_{n}
\rightarrow E_{n-1} \rightarrow ... \rightarrow E_{1} \rightarrow
E_{0} \rightarrow M \rightarrow 0.$$ If $n$ is the
least such integer, $\wdim_{R}(M)=n$ (resp., $\pdim_{R}(M)=n$). If no such resolution exists,
$\wdim_{R}(M)=\infty$ (resp., $\pdim_{R}(M)=\infty$). The weak
global dimension (resp., global dimension) of $R$, denoted by
$\wgdim(R)$ (resp., $\gdim(R)$), is the supremum of $\wdim_{R}(M)$
(resp., $\pdim_{R}(M)$), where $M$ ranges over all (finitely
generated) R-modules. For more details on all these notions, we refer the
reader to \cite{B,G1,Ro}.

A ring $R$ is called coherent if every finitely
generated ideal of $R$ is finitely presented; equivalently, if
$(0:a)$ and $I \cap J$ are finitely generated for every $a \in R$
and any two finitely generated ideals $I$ and $J$ of $R$ \cite{G1}. Examples
of coherent rings are Noetherian rings, Boolean algebras, von
Neumann regular rings, and semihereditary rings.

Gaussian rings belong to the class of Pr\"ufer-like rings which has recently received much attention
from commutative ring theorists. A ring $R$ is called Gaussian if
for every $f, g \in R[X]$, one has the content ideal equation
$c(fg) = c(f)c(g)$ where $c(f)$, the content of $f$, is the ideal
of $R$  generated by the coefficients of $f$ \cite{T}. The ring
$R$ is said to be a chained ring (or valuation ring) if its
lattice of ideals is totally ordered by inclusion; and $R$ is
called arithmetical if $R_{\m}$ is a chained ring for each maximal
ideal $m$ of $R$ \cite{Fu,J}. Also $R$ is called semihereditary if
every finitely generated ideal of $R$ is projective \cite{CE}; and
$R$ is Pr\"ufer if every finitely generated regular ideal of $R$
is projective \cite{BS,Gr}. In the domain context, all these
notions coincide with the concept of Pr\"ufer domain. Glaz, in
\cite{G3}, constructs examples which show that all these notions
are distinct in the context of arbitrary rings. More examples, in
this regard, are provided via trivial ring extensions \cite{AJK,BKM}.

The following diagram of implications puts the notion of Gaussian
ring in perspective within the family of Pr\"ufer-like rings
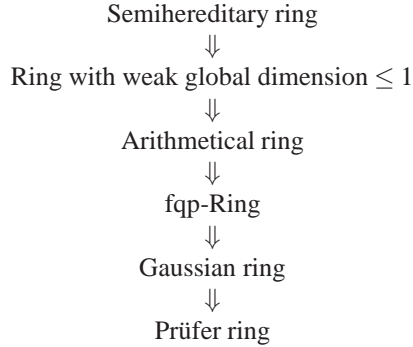
\cite{BG,BG2,AJK}:\bigskip

\begin{center}
Semihereditary ring\\
$\Downarrow$\\
Ring with weak global dimension $\leq 1$\\
$\Downarrow$\\
Arithmetical ring\\
$\Downarrow$ \\
fqp-Ring\\
$\Downarrow$\\
Gaussian ring\\
$\Downarrow$\\
Pr\"ufer ring
\end{center}

In 1969, Osofsky proved that a local arithmetical ring (i.e.,
chained ring) with zero divisors has infinite weak global
dimension \cite{Os}. In view of \cite[Corollary 4.2.6]{G1}, this
results asserts that the weak global dimension of an arithmetical
ring is 0, 1, or $\infty$.

In 2007, Bazzoni and Glaz proved that if $R$ is a coherent
Pr\"ufer ring (and, a fortiori, a Gaussian ring), then $\wgdim(R)$
= 0, 1, or $\infty$ \cite[Proposition 6.1]{BG2}. And also they
proved that if $R$ is a Gaussian ring admitting a maximal ideal
$\m$ such that the nilradical of the localization $R_{m}$ is a
nonzero nilpotent ideal. Then $\wgdim(R)$ = $\infty$ \cite[Theorem 6.4]{BG2}. At the end of the paper, they conjectured that ``the weak global dimension of a Gaussian ring is 0, 1, or
$\infty$" \cite{BG2}. In two preprints \cite{DT1,DT2}, Donadze and Thomas claim to prove this conjecture (see the end of Section \ref{G}).

In 2010, the authors of \cite{BKM} proved that if $(A,\m)$ is a
local ring, $E$ is a nonzero $\frac{A}{\m}$-vector space, and
$R:=A\ltimes~E$ is the trivial extension of $A$ by $E$, then:
\begin{itemize}
\item $R$ is a total ring of quotients and hence a Pr\"ufer ring.
\item $R$ is Gaussian if and only if $A$ is Gaussian.
\item $R$ is arithmetical if and only if $A:=K$ is a field and $\dim_{K}E=1$.
\item $\wgdim(R)\gvertneqq 1$. If, in addition, $\m$ admits a minimal generating set, then\break $\wgdim(R)=\infty$.
\end{itemize}
As an application, they provided an example of a Gaussian ring
which is neither arithmetical nor coherent and has an infinite
weak global dimension \cite[Example 2.7]{BKM}; which widened the
scope of validity of the above conjecture beyond the class of
coherent Gaussian rings.

In 2011, the authors of \cite{AJK} investigated the correlation
of fqp-rings with well-known Pr\"ufer conditions; namely, they
proved that the class of fqp-rings stands between the two classes
of arithmetical rings and Gaussian rings \cite[Theorem 3.1]{AJK}.
They also examined the transfer of the fqp-property to trivial
ring extensions in order to build original examples of fqp-rings.
Also they generalized Osofsky's result (mentioned above) and
extended Bazzoni-Glaz's result on coherent Gaussian rings by
proving that the weak global dimension of an fqp-ring is equal to
0, 1, or  $\infty$ \cite[Theorem 3.11]{AJK}; and then they
provided an example of an fqp-ring that is neither arithmetical
nor coherent \cite[Example 3.9]{AJK}.

Recently, several papers have appeared in the literature investigating the weak global dimension of various settings subject to Pr\"ufer conditions.  This survey paper plans to track and study these works dealing with this topic from the very origin; that is, 1969 Osofsky's proof of the existence of a module with infinite projective dimension on a local arithmetical ring. Precisely, we will examine all main results published in \cite{AJK,BKM,BG2,G2,Os}.

Our goal is to make this topic accessible and appealing to a broad audience; including graduate students. For this purpose, we present complete proofs of all main results via ample details and simplified arguments along with exact references. Further, numerous examples -arising as trivial ring extensions- are provided to illustrate the concepts and results involved in this paper. We assume familiarity with the basic tools used in the homological aspects of commutative ring theory, and any unreferenced material is standard as in \cite{AM,B,CE,G1,H,L,Ro,V}.
\end{section}

\begin{section}{Weak global dimension of arithmetical rings}\label{A}

\noindent In this section, we provide a detailed proof for Osofsky's Theorem
that the weak global dimension of an arithmetical ring with zero
divisors is infinite. In fact, this result enables one to state
that the weak global dimension of an arithmetical ring is $0$,
$1$, or $\infty$. We start by recalling some basic definitions.

\begin{definition}\rm \label{A1} Let $R$ be a ring and $M$ an $R$-module. Then:

\begin{enumerate}
\item The weak dimension of $M$, denoted by $\wdim(M)$, measures
how far $M$ is from being flat. It is defined as follows: Let $n$
be a positive integer. We have $\wdim(M)\leq n$ if there is a flat
resolution
$$0\rightarrow E_{n} \rightarrow E_{n-1} \rightarrow ...
\rightarrow E_{1} \rightarrow E_{0} \rightarrow M \rightarrow 0.$$
If no such resolution exists, $\wdim(M)=\infty$; and if $n$ is the
least such integer, $\wdim(M)=n$.

\item The weak global dimension of $R$, denoted by $\wgdim(R)$, is
the supremum of $\wdim(M)$, where $M$ ranges over all (finitely
generated) R-modules.\end{enumerate}
\end{definition}

\begin{definition}\rm \label{A2} Let $R$ be a ring. Then:
\begin{enumerate}
\item $R$ is said to be a chained ring (or valuation ring) if its
lattice of ideals is totally ordered by inclusion.

\item $R$ is called an arithmetical ring if if $R_{\m}$ is a chained
ring for each maximal ideal $m$ of $R$. \end{enumerate}
\end{definition}

Fields and $\Z_{(p)}$, where $\Z$ is the ring of integers and $p$
is a prime number, are examples of chained rings. Also,
$\Z/n^{2}\Z$ is an arithmetical ring for any positive integer $n$.
For more examples, see {\cite {BKM}}. For a ring $R$, let $\Ze(R)$
denote the set of all zero divisors of $R$.

Next we give the main theorem of this section.

\begin{thm}\label{A3}
Let $R$ be an arithmetical ring. Then $\wgdim(R)=$ 0, 1, or
$\infty$.
\end{thm}

To prove this theorem we make the following reductions:

\noindent (1) We may assume that $R$ is a chained ring since $\wgdim(R)$ is
the supremum of $\wgdim(R_{ \m})$ for all  maximal ideal $ \m$ of
$R$\ {\cite [Theorem 1.3.14 (1)]{G1}}.

\noindent (2) We may assume that $R$ is a chained ring with zero divisors.
Then we prove that $\wgdim(R)=\infty$ since, if $R$ is a valuation
domain, then  $\wgdim(R)\leq 1$ by {\cite [Corollary 4.2.6] {G1}}.

\noindent (3) Finally, we may assume that  $(R, \m)$ is a chained ring with
zero divisors such that $\Ze(R) =\m$,  since $\Ze(R)$ is a prime
ideal, $\Ze(R_{\Ze(R)})=\Ze(R)R_{\Ze(R)}$, and
\\ $\wgdim(R_{\Ze(R)}) \leq \wgdim(R)$.

So our task is reduced to prove the following theorem.

\begin{thm}[{\cite[Theorem]{Os}}]\label{A4}
Let $(R, \m)$ be a chained ring with zero divisors such that
$\Ze(R) =\m$. Then $\wgdim(R)=\infty$.
\end{thm}

To prove this theorem we first prove the following lemmas.
Throughout, let $(R, \m)$ be a chained ring with $\Ze(R) =\m$, $M$
an $R$-module, $I=\{x \in R \mid x^{2}=0\}$, and for $x\in M$,
$(0:x)=\{y \in R \mid yx=0\}$. One can easily check that $I$ is a
nonzero ideal since $R$ is a chained ring with zero divisors.

\begin{lemma}[{\cite[Lemma 1]{Os}}] \label{A5} $I^{2}=0$, and for all $x\notin R$, $x\notin I$
$\Rightarrow$ $(0:x)\subseteq I$. \end{lemma}

\begin{proof}
To prove that $I^{2}=0$, it suffices to prove that $ab=0$ for all
$a,b \in I$. So let $a, b \in I$. Then either $a\in bR$ or $b\in
aR$, so that $ab\in a^{2}R =0$ or $ab\in b^{2}R =0$.

 Now let $x\in R\setminus I$ and  $y\in (0:x)$. Then either $x\in yR$ or $y\in xR$. But $x\in yR$ implies that $x^{2}\in
xyR =0$, absurd. Therefore $y\in xR$, so that $y^{2}\in xyR=0$.
Hence $y\in I$. \end{proof}

\begin{lemma}[{\cite[Lemma 2]{Os}}] \label{A6} Let $0\neq x\in \Ze(R)$ such that $(0:x)=yR$. Then\\
$\wgdim(R)=\infty$.\end{lemma}

\begin{proof}
We first prove that $(0:y)=xR$. The inclusion $(0:y)$$\supseteq$
$xR$ is trivial since $xy=0$. Now to prove the other inclusion let
$z\in (0:y)$. Then either $z=xr$ for some $r\in R$ and in this
case we are done, or $x=zj$ for some $j\in R$. We may assume $j\in
\m$. Otherwise, $j$ is a unit and then we return to the first
case. Since $x\neq 0$, $j\notin (0:z)$, so $jR\nsubseteq (0:z)$
which implies $(0:z) \subseteq jR$, and hence  $y=jk$ for some $k
\in \m$. But then $0=zy=zjk=xk$, so $k \in (0:x) = yR$, and hence
$k=yr$ for some $r\in R$. Hence $y=kj=yrj$, and as $j\in \m$ we
have the equality $y=y(1-rj)(1-rj)^{-1}=0$, which contradicts the
fact that  $x$ is a zero divisor. Hence $z\in xR$, and therefore
$(0:y)=xR$.

Now let $m_{x}$ (resp., $m_{y}$) denote the multiplication by $x$
(resp., $y$). Since $(0:x)=yR$ and $(0:y)=xR$ we have the
following infinite flat resolution of $xR$ with syzygies $xR$ and
$yR$:

$$\begin{array}{ccccccccccccccc}
...\longrightarrow R\overset{m_{y}}{\longrightarrow
}R\overset{m_{x}}{\longrightarrow} R \overset{m_{y}}
{\longrightarrow}... \overset{m_{y}} {\longrightarrow }R
\overset{m_{x}}{\longrightarrow} xR {\longrightarrow} 0

\end{array}$$
We claim that $xR$ and $yR$ are not  flat. Indeed, recall that a
projective module over a local ring is free \cite{Ro}. So no
projective module is annihilated by $x$ or $y$. Since  $xR$ is
annihilated by $y$ and $yR$ is annihilated by $x$, both $xR$ and
$yR$ are not projective. Further,  $xR$ and $yR$ are finitely
presented in view of the exact sequence $0\rightarrow yR
\rightarrow R \rightarrow xR \rightarrow 0$. It follows that $xR$
and $yR$ are not flat (since a finitely presented  flat module is
 projective {\cite[Theorem 3.61]{Ro}}).
\end{proof}

\begin{corollary}[{\cite[Corollary]{Os}}] \label{A7} If $I= \m$, then I is cyclic and $R$ has infinite
weak global dimension.
\end{corollary}

\begin{proof}
Assume that $I = \m$. Then $\m^{2} =0$. Now let $0\neq a \in \m$.
We claim that  $\m= aR$. Indeed, let $b\in \m$. Since $R$ is a
chained ring, either $b = ra$ for some $r\in R$ and in this case
we are done, or $a=rb$ for some $r\in R$. In the later  case,
either $r$ is a unit and then $b=r^{-1}a \in aR$, or $r\in \m$
which implies  $a=rb=0$, which contradicts the assumption $a\neq
0$. Thus  $\m= aR$, as claimed. Moreover, we have $(0:a) = aR$.
Indeed, $(0:a) \supseteq aR$ since $a \in I$; if $x\in (0:a)$,
then $x\in \Ze(R) = \m = aR$. Hence $(0:a) = aR$. It follows that
$R$ satisfies the conditions of Lemma \ref{A6} and hence the weak
global dimension of $R$ is $\infty$.
\end{proof}

Throughout, an element $x$ of  an $R$- module $M$ is said to be
regular if $(0:x)=0$.

\begin{lemma}[{\cite[Lemma 3]{Os}}] \label{A8}  Let $F$ be a free module and $x \in F$. Then $x$ is
contained in $zR$ for some regular element $z$ of  $F$.
\end{lemma}

\begin{proof}
Let $\{y_{\alpha}\}$ be a basis for $F$ and let $x:=
\displaystyle \sum \limits_{ i=1}^{n}y_{i}r_{i}\in F$, where
$r_{i} \in R$. Since $R$ is a chained ring, there is $j \in \{1,2,
..., n\}$ such that $\displaystyle \sum \limits_{ i=1}^{n}
r_{i}R\subseteq r_{j}R$. So that for each $i \in \{1,2, ..., n\}$,
$r_{i} = r_{j}s_{i}$ for some $s_{i}\in R$ with $s_{j} = 1$. Hence
$\displaystyle x = r_{j}(\sum \limits_{i=1}^{n} (y_{i}s_{i}))$. We claim that $\displaystyle z:=\sum \limits_{i=1}^{n} y_{i}s_{i}$
is regular. Suppose not and let $t\in R$ such that $\displaystyle t(\sum
\limits_{i=1}^{n} y_{i}s_{i}) = 0$. Then $ts_{i} = 0$ for all  $i
\in\{1,2, ..., n\}$. In particular $t=ts_{j} = 0$, absurd.
Therefore $z$ is regular and $x=r_{j}z$, as desired.
\end{proof}

Note, for convenience, that in the proof of Theorem~\ref{A4} (below) we will prove the existence of a module $M$ satisfying the conditions (1) and (2) of the next lemma; which will allow us to construct -via iteration- an infinite flat resolution of $M$. 

\begin {lemma}[{\cite[Lemma 4]{Os}}] \label{A9}Assume that $(0:r)$ is infinitely
generated for all $0\neq r\in \m$. Let $M$ be an $R$-submodule of
a free module $N$ such that:

$(1)$  $M= M_{1}\bigcup M_{2} \bigcup M_{3}$, where $M_{1} =
\displaystyle \bigcup _{\substack{
           x\in M\\
            x\ regular}} xR$, $ M_{2}=
\displaystyle \bigcup \limits_{i=0}^{\infty}yu_{i}R$, with $y$
regular in $N$, $u_{i}R \subsetneqq u_{i+1}R$, and $yu_{i}$ is not
in $M_{1}$, and $M_{3} = \displaystyle \sum  v_{j}R$.

 $(2)$ $yu_{0}R \displaystyle \cap xR$ is infinitely generated for some regular $x\in M$.

Let $F$ be a free $R$-module with basis $\{y_{x}\ |\ x\ regular
\in M\} \cup \{z_{i}\ |\ i\in\omega\} \cup \{w_{j}\}$, and let
$v:F\longrightarrow N$ be the map defined by: $v(y_{x})= x$,
$v(z_{i}) = yu_{i}$, and    $v(w_{j})= v_{j}$. Then $K=Ker(v)$ has
properties $(1), (2)$, and $M$ is not flat. \end {lemma}

\begin{proof}
First the map $v$ exists by {\cite[Theorem 4.1]{L}}. (1) By (2),
there exist $r, s \in R$ such that $yu_{0}r = xs \neq 0$. Here
$r\in \m$; otherwise, $yu_{0} = xsr^{-1} \in M_{1}$,
contradiction. Since $\Ze(R) =\m$, the expression for any regular
element in terms of a basis for $N$ has one coefficient a unit.
Indeed, let $(n_{\alpha})_{\alpha\in \Delta}$ be a basis for $N$
and $z$ a regular element in $N$ with $z= \displaystyle \sum
\limits _{i=0}^{i=k} c_{i}n_{i}$ where $c_{i}\in R$. As $R$ is a
chained ring, there exists $j\in \{0, ... , k\}$ such that for all
$i \in \{0, ... , k\}$, there exists  $d_{i} \in R$ with $c_{i}=
c_{j}d_{i}$ and $d_{j}=1$. We claim that $c_{j}$ is a unit. Suppose not.
Then $c_{j}\in Z(R)$. So there is a nonzero $d \in R$ with
$dc_{j}=0$, and hence $ \displaystyle dz= dc_{j} \sum \limits
_{i=0}^{i=k}d_{i}n_{i} =0$. This is absurd since $z$ is regular.

 Now, let $\displaystyle x=\sum _{\substack{
           i \in I\\
            I\ finite}}
a_{i}n_{i}$ and $\displaystyle y=\sum  _{\substack{
           i \in I\\
            I\ finite}}
b_{i}n_{i}$. Then $b_{i}u_{0}r = a_{i}s$ for all $i\in I$. Let
$i_{0} \in I$ such that $a_{i_{0}}$ is a unit. So $s= u_{0}rt$,
where  $t = b_{i_{0}}a_{i_{0}}^{-1} \in R$. Note that
$b_{i_{0}}\neq 0$ since $xs \neq 0$. Clearly, $z_{0}-y_{x}u_{0}t$
is regular in $F$ (since $z_{0}, y_{x}$ are part of the basis of
$F$), is not in $K$ (otherwise, $v(z_{0}-y_{x}u_{0}t)=0$ yields
$yu_{0}=xu_{0}t$, which contradicts $(1)$), and
$(z_{0}-y_{x}u_{0}t)r\in K$. We claim that $(z_{0}-y_{x}u_{0}t)r$
is not in $\displaystyle K_{1} := \bigcup  _{\substack{
           x^{\prime} \in K\\
            x^{\prime}\ regular}}
x^{\prime}R$. Suppose not and  assume that $r(z_{0}-u_{0}ty_{x})=
r^{\prime}x^{\prime}$ with $r^{\prime}\in R$ and  $x^{\prime}$
regular in $K$. Then $r^{\prime} \neq 0$ since $r \neq 0$ and as
$x^{\prime} \in K \subseteq F$, there are $a, b, a_{i}\in R$ such
that $x^{\prime}= az_{0}-by_{x}+ x^{\prime\prime}$, where
$\displaystyle x^{\prime\prime} = \sum _{\substack{
           y_{x} \neq f_{i}\\
           z_{0} \neq f_{i}}} a_{i}f_{i}$. Thus $r=r^{\prime}a$,
           $ru_{0}t=r^{\prime}b$,
 and $r^{\prime}x^{\prime\prime}=0$. Since $x^{\prime}$ is regular
in $F$ and $r^{\prime}x^{\prime\prime}=0$, $a$ or $b$ is unit. We
claim that $a$ is always a unit. Indeed, if $b$ is a unit, then
$r(1 - ab^{-1}u_{0}t)=0$, so if $a \in \m $, then $(1 -
ab^{-1}u_{0}t)$ is a unit which implies  $r=0$, absurd. So
$a^{-1}x^{\prime}= z_{0}-a^{-1}by_{x}+ a^{-1}x^{\prime\prime}$,
$r^{\prime}=a^{-1}r$, and $ru_{0}t=ra^{-1}b$ which implies $z_{0}-
u_{0}ty_{x}+(u_{0}t -
a^{-1}b)y_{x}+a^{-1}x^{\prime\prime}=a^{-1}x^{\prime} \in K$. By
Lemma \ref{A8} $(u_{0}t - a^{-1}b)y_{x}+a^{-1}x^{\prime\prime} =
pq$, fore some  $q$ regular in $F$ and $p\in R$. But clearly since
$r=r^{\prime}a$, $ru_{0}t=r^{\prime}b$, and
$r^{\prime}x^{\prime\prime}=0$, then  $rpq=0$. Hence $rp=0$. It
follows that $(z_{0} - y_{x}u_{0}t + qp)\in K$, where $q$ is
regular in $F$ and $ p\in (0:r)$. Thus by applying $v$ we obtain
$yu_{0}-xu_{0}t + pv(q) = 0$. But $R$ is a chained ring, so  $p$
and $u_{0}t$ are comparable and since $u_{0}tr \neq 0$,
$p=u_{0}th$ for some $h\in R$. Hence $yu_{0} = ( x -
hv(q))u_{0}t$, we show that $( x - hv(q))$ is regular in $M$ which
contradicts property $(1)$. First clearly $( x - hv(q))\in M$
since $x, v(q) \in M$. Now suppose that $a( x - hv(q))=0$ for some
$a\in \m$. Either $u_{0}t=a^{\prime}a$ for some $a^{\prime}\in R$,
this yields $yu_{0}=( x - hv(q))aa^{\prime}=0$ also impossible, or
$a=u_{0}tm$ for some $m\in R$, and this yields
$mu_{0}y=(x-hv(q))a=0$, so $mu_{0}=0$ as $y$ is regular, and hence
$a=mu_{0}t=0$. We conclude that $( x - hv(q))$ is regular in $M$
and hence $yu_{0} \in M_{1}$, the desired contradiction.

Last, let $yu_{0}R$ $\cap$ $xR$ $=$ $\langle x_{0}, x_{1},
...,x_{n}, ... \rangle$, where $$\langle x_{0}, x_{1}, ... ,x_{i}
\rangle \subsetneqq \langle x_{0}, x_{1}, ...,x_{i}, x_{i+1}
\rangle.$$ For any integer $i\geq 0$, let $x_{i}=yu_{0} r_{i}$ for
some $r_{i}\in R$. It is clear that $r_{0}R$ $\subsetneqq$
$r_{1}R$ $\subsetneqq$ ... $\subsetneqq$ $r_{i}R$ $\subsetneqq$
$r_{i+1}R$ $\subsetneqq$ ... . Now, let $y^{\prime} := z_{0} -
y_{x}u_{0}t$, $u_{i}^{\prime} := r_{i}$ for each $i\in
\mathbb{N}$. Then $K=K_{1}\bigcup K_{2} \bigcup K_{3}$, where
$\displaystyle K_{1} := \bigcup _{\substack{
           x^{\prime} \in K\\
            x^{\prime}\ regular}}x^{\prime}R$, $K_{2}:= \bigcup \limits
_{i=0}^{\infty}y^{\prime}u_{i}^{\prime}R$ with $y^{\prime}$
regular in $F$ and $u_{i}^{\prime}R \subsetneqq
u_{i+1}^{\prime}R$, and $\displaystyle K_{3}:= K\setminus
(K_{1}\bigcup K_{2})$. Thus $K$ satisfy Property (1).

(2) Since $u_{0}R\subsetneqq u_{1}R$, $u_{0} = u_{1}m^{\prime}$
for some $m^{\prime}\in \m$. Hence $x^{\prime}:= z_{0} -
z_{1}m^{\prime}$ is regular in $K$ since $v(x^{\prime})= v(z_{0} -
z_{1}m^{\prime})=yu_{0} - yu_{1}m^{\prime}=  0$ and $z_{0}, z_{1}$
are basis elements. We claim that $(z_{0} - z_{1}m^{\prime})R$
$\cap$ $(z_{0}-y_{x}u_{0}t)r_{0}R = z_{0}(0:m^{\prime})$. Indeed,
since $z_{0}, z_{1}, y_{x}$ are basis elements, then $(z_{0} -
z_{1}m^{\prime})R$ $\cap$ $(z_{0}-y_{x}u_{0}t)r_{0} \subseteq
z_{0}R$. Also $(z_{0} - z_{1}m^{\prime})R \cap z_{0}R =
z_{0}(0:m^{\prime})$. For, let $ l\in (z_{0} - z_{1}m^{\prime})R
\cap z_{0}R$. Then $l= (z_{0} -z_{1}m^{\prime})a =
z_{0}a^{\prime}$ for some $a, a^{\prime} \in R$. Hence $a =
a^{\prime}$ and $am^{\prime} = 0$, whence $l = az_{0}$ with
$am^{\prime} = 0$. So $l \in z_{0}(0:m^{\prime})$. The reverse
inclusion is straightforward. Consequently, $(z_{0} -
z_{1}m^{\prime})R$ $\cap$ $(z_{0}-y_{x}u_{0}t)r_{0}R$ $\subseteq$
$z_{0}(0:m^{\prime})$. To prove the reverse inclusion, let $k \in
(0:m^{\prime})$. Then either $k=r_{0}k^{\prime}$  or $r_{0}
=kk^{\prime}$, for some $k^{\prime} \in R$. The second case is
impossible since $r_{0}u_{0}\neq 0$. Hence $z_{0}k=
(z_{0}-y_{x}u_{0}t)r_{0}k^{\prime} \in (z_{0}-y_{x}u_{0}t)r_{0}R
$. Further,  $z_{0}k \in (z_{0} - z_{1}m^{\prime})R$. Therefore
our claim is true. But $z_{0}$ is regular, so $z_{0}(0:m^{\prime})
\cong (0:m^{\prime})$ which is infinitely generated by hypothesis.
Therefore $y^{\prime}u_{0}^{\prime}R \cap x^{\prime}R$ is
infinitely generated, as desired.

Finally, $M$ is not flat. Suppose not, then by {\cite[Theorem 3.57]{Ro}},
there is an $R$-map $\theta: F\longrightarrow K$ such that
$\theta$ $(( z_{0}-y_{x}u_{0}t)r_{0})$ $=$
$(z_{0}-y_{x}u_{0}t)r_{0}$. Assume that
$\theta(z_{0})=az_{0}+by_{x}+Z_{1}$ for some $a, b\in R$ and
$\theta(y_{x})=a^{\prime}z_{0}+b^{\prime}y_{x}+Z_{2}$ for some
$a^{\prime}, b^{\prime} \in R$. Then
$r_{0}a-r_{0}u_{0}ta^{\prime}=r_{0}$,
$r_{0}b-r_{0}u_{0}tb^{\prime}=-r_{0}u_{0}t$, and
$r_{0}Z_{1}-r_{0}u_{0}tZ_{2}=0$. Hence
$r_{0}(1-a+u_{0}ta^{\prime})=0$ and since $r_{0}\neq 0$, $a$ or
$a^{\prime}$ is a unit. Suppose that  $a$ is a unit and without
loss of generality we can  assume that $a=1$. Thus we have the
equation $z_{0}-u_{0}ty_{x} -
u_{0}ta^{\prime}z_{0}+(u_{0}t-u_{0}tb^{\prime}+b)y_{x}+Z_{1}-u_{0}tZ_{2}=\theta(z_{0})-u_{0}t\theta(Z_{2})
\in K$. By Lemma \ref{A8},
$-u_{0}ta^{\prime}z_{0}+(u_{0}t-u_{0}tb^{\prime}+b)y_{x}+Z_{1}-u_{0}tZ_{2}
= pq$, where $q$ is regular in $F$ and, clearly, $r_{0}p=0$ since
$r_{0}u_{0}ta^{\prime}=0$. Thus $z_{0}-u_{0}ty_{x} + pq \in K$, which is
absurd (as seen before in the second paragraph of  the proof of
Lemma \ref{A9}).
\end {proof}

Now we are able to prove Theorem \ref {A4}.
\bigskip

\textbf{Proof of Theorem \ref {A4}.}
If $(0:r)$ is cyclic for some $r \in \m$, then $R$ has infinite
weak global dimension by Lemma \ref {A6}. Next suppose that
$(0:r)$ is not cyclic, for all $0\neq r\in \m$. Which is
equivalent to assume that $(0:r)$ is infinitely generated for all
$0\neq r\in \m$, since $R$ is a chained ring.

Let $0\neq a\in I$ and $b\in \m \setminus I$. Note that $b$ exists
since $I\neq \m$  by the proof of Corollary \ref{A7}. Let $N$ be a
free $R$-module on two generators $y, y^{\prime}$ and let $M := (y
- y^{\prime}b)R + y(0:a)$. Then:

(A) $\displaystyle M_{1} := \bigcup _{\substack{
           x\in M\\
            x\ regular}} xR
= \{(yt - y^{\prime}b)r | 1 - t \in (0:a), r\in R)$. To show this
equality, let $c$ be a regular element in $M$. Then $c = (r_{1} +
r_{2})y - r_{1}by^{\prime}$ for some $r_{1} \in R, r_{2}\in
(0:a)$. We claim that $r_{1}$ is a unit. Suppose not. So either $r_{1}\in
(r_{2})$ hence $ac = 0$, or $r_{2} = nr_{1}$ for some $n\in R$ and
since $r_{1} \in \m = \Ze(R)$, there is $r_{1}^{\prime}\neq 0$
such that $r_{1}r_{1}^{\prime} =0$, so $r_{1}^{\prime}c = 0$. In
both cases there is a contradiction with the fact that $c$ is
regular. Thus, $r_{1}$ is a unit. It follows that $c = (1 +
r_{1}^{-1} r_{2})yr_{1} - by^{\prime}r_{1} \in \{(yt -
y^{\prime}b)r | 1 - t \in (0:a), r\in R\}$. Now let $c = yt -
y^{\prime}b$, where $(1-t)\in (0:a)$. Then $c$ is regular. Indeed,
if $rc=0$ for some $r\in R$, then $rt=0$. Moreover, either $r=na$
for some $n\in R$, and in this case $r(1-t)= na(1-t)=0$, so
$r=rt=0$ as desired, or $a=nr$ for some $n\in R$, so $a=
at=nrt=0$, absurd.

(B) There exists a countable chain of ideals $u_{0}R \subsetneqq
u_{1}R \subsetneqq ...$ where $u_{i} \in (0:a) \setminus (0:b)$.
Since $0\neq a\in I$ and $b\in \m \setminus I$, $(a) \subseteq
(b)$. Thus $(0:b)\subseteq (0:a)$. Moreover $(0:b)\subsetneqq
(0:a)$; otherwise, $a\in (0:a) = (0:b)$, and hence $ab=0$. Hence
$b\in (0:a)=(0:b)\subseteq I$ by Lemma \ref{A5}, absurd. Now let
$u_{0} \in (0:a) \setminus (0:b)$. Since $(0:a)$ is infinitely
generated, there are $u_{1}, u_{2}, ...$ such that $(u_{0})
\subsetneqq (u_{0},u_{1}) \varsubsetneqq ... \subseteq (0:a)$. So
$u_{0}R \varsubsetneqq u_{1}R \varsubsetneqq ...$ and necessarily
$u_{i} \notin (0:b)$ for all $i\geq 1$ since $u_{0} \notin (0:b)$.

Note that $yu_{i} \in M$(since $u_{i} \in (0:a)$). Also
$yu_{i}\notin M_{1}$; otherwise, if  $yu_{i} = ytr - y^{\prime}br$
with $1-t \in (0:a)$ and $r \in R$, then $u_{i} = tr$ and $br=0$.
Hence $bu_{i}= btr =0$ and thus $u_{i} \in (0:b)$, contradiction.
Also note that $y$ is regular in $N$ (part of the basis) and
$y\notin M$; if $y=(y-y^{\prime}b)r_{1}+r_{2}y$ with $r_{1}\in R$
and $r_{2}\in (0:a)$, then $r_{1}b = 0$ and $r_{1}+r_{2} = 1$. So
$r_{1}\in \m$, $ar_{1}=a$, and hence $a=0$, absurd.

(A) and (B) imply that (1) of Lemma \ref{A9} holds.

Let us show that $yu_{0}R \cap (y-y^{\prime}b)R = y(0:b)$. Indeed,
if $c= yu_{0}r = (y - y^{\prime}b)r^{\prime}$ where $r, r^{\prime}
\in R$, then $u_{0}r = r^{\prime}$ and $r^{\prime}b=0$. Hence
$c\in y(0:b)$. If $c=ry$ where $rb=0$, then $r=u_{0}t$ for some
$t\in R$ as $u_{0} \in (0:a)\setminus (0:b)$. Thus
$c=r(y-y^{\prime}b)$. Now  $y(0:b) \cong (0:b)$ is infinitely
generated. Therefore (2) of Lemma \ref{A9} holds.

Since $K$ satisfies the properties of $M$ we can consider it as a
new module $M$, and then there is a free module $F_{1}$ and a map
$v_{1}:F_{1}\longrightarrow F$ such that $K_{1}=Ker(v_{1})$
satisfies the same conditions of $K$ and $K_{1}$ is not flat. We
can repeat this iteration above to get the infinite flat
resolution of $M$: \begin{center}$...\rightarrow F_{n} \rightarrow
F_{n-1} \rightarrow ... \rightarrow F_{1} \rightarrow F_{0}
\rightarrow M \rightarrow 0$. \end{center}
 with none of
the syzygies $K, K_{1}, K_{2}, ... $ is flat. Therefore $R$ has an
infinite weak global dimension. $\Box$

\end{section}

\begin{section}{Weak global dimension of Gaussian rings}\label{G}

\noindent In 2005, Glaz proved that if $R$ is a Gaussian coherent ring, then
$\wgdim(R)=$ 0, 1, or $\infty$ \cite {G2}. In this section, we
will see that the same conclusion holds for the larger class of
Pr\"ufer coherent rings and fore some contexts of Gaussian rings.
We start by recalling the definitions of Gaussian, Pr\"ufer, and
coherent rings.

\begin{definition} \rm  \label{B1} Let $R$ be a ring. Then:
\begin{enumerate}

\item $R$ is called a Gaussian ring if for every $f, g \in R[X]$, one
has the content ideal equation $c(fg) = c(f)c(g)$, where $c(f)$,
the content of $f$, is the ideal of $R$  generated by the
coefficients of $f$.

\item $R$ is called a Pr\"ufer ring if every nonzero finitely generated regular ideal is
invertible (or, equivalently, projective)

\item $R$ is called a coherent ring if every finitely generated ideal
of $R$ is finitely presented; equivalently, if $(0:a)$ and $I \cap
J$ are finitely generated for every $a \in R$ and any two finitely
generated ideals $I$ and $J$ of $R$. \end{enumerate}
\end{definition}

Recall that Arithmetical ring $\Rightarrow$ Gaussian ring
$\Rightarrow$ Pr\"ufer ring. To see the proofs of the above
implications and that they cannot be reversed, in general, we
refer the reader to \cite { BG2, G2, G3} and Section 5 of this
paper.

Noetherian rings, valuation domains, and $K[x_{1},x_{2}, ... ]$
where $K$ is a field are examples of coherent rings. For more
examples, see \cite {G1}.

 Let $Q(R)$ denote the total ring of
fractions of $R$ and $\Nil(R)$ its nilradical. The following
proposition is the first main result of this section.
\begin{proposition} [{\cite[proposition 6.1]{BG2}}]\label{B2} Let $R$ be a coherent Pr\"ufer ring. Then
the weak global dimension of $R$ is equal to 0, 1, or $\infty$. \end{proposition}

The proof of this proposition relies on  the following lemmas.
Recall that a ring  $R$ is called regular if every finitely
generated ideal of $R$ has a finite projective dimension; and von
Neumann regular if every $R$-module is flat.

\begin{lemma}[{\cite[Corollary 6.2.4]{G1}}] \label{B3} Let $R$ be a coherent regular ring. Then $Q(R)$ is a von Neumann regular ring.
$\Box$ \end{lemma}

\begin{lemma} [{\cite[Lemma 2.1]{G2}}]  \label{B4} Let $R$ be a
local Gaussian ring and $I=(a_{1}, ..., a_{n})$ be a finitely
generated ideal of $R$. Then $I^{2}=(a_{i}^{2})$, for some $ i \in
\{1,2, ..., n\}$. \end{lemma}

\begin{proof}
We first assume that $I=(a, b)$. Let $f(x):=ax+b$, $g(x):=ax-b$,
and $h(x):=bx+a$. Since $R$ is Gaussian, $c(fg)=c(f)c(g)$, so that
$(a, b)^{2}=(a^{2}, b^{2})$, also $c(fh)=c(f)c(h)$ which implies
that $(a, b)^{2}=(ab,a^{2}+b^{2})$. Hence $(a^{2},
b^{2})=(ab,a^{2}+b^{2})$, whence $a^{2}=rab+s(a^{2}+b^{2})$, for
some $r$ and $s$ in $R$. That is, $(1-s)a^{2}+rab+sb^{2}=0$. Since
$R$ is a local ring, either $s$ or $1-s$ is a unit in $R$. If $s$
is a unit in $R$, then $b^{2}+rs^{-1}ab+(s^{-1}-1) a^{2}=0$. Next
we show that $ab \in (a^{2})$. Let $k(x):=(b+\alpha a)x-a$, where
$\alpha:=rs^{-1}$. Then $c(hk)=c(h)c(k)$ implies that
$\displaystyle (b(b+\alpha a), \alpha a^{2}, -a^{2})=(a, b)((b+
\alpha a), a)$. But clearly $(b(b+\alpha a), \alpha a^{2},
-a^{2})=((s^{-1}-1) a^{2}, \alpha a^{2}, -a^{2})=(a^{2})$. Thus
$(a^{2})=(a, b)((b+ \alpha a), a)$. In particular, $ab \in
(a^{2})$ and so does $b^{2}$. If $1-s$ is unit, similar arguments
imply that $ab$, and hence $a^{2}\in (b^{2})$. Thus for any two
elements $a$ and $b$, $ab \in (b^{2})$ or $(a^{2})$. It follows
that $I^{2}=(a_{1}, ..., a_{n})^{2}=(a_{1}^{2}, ..., a_{n}^{2})$.
An induction on $n$ leads to the conclusion.
\end{proof}

Recall that a ring $R$ is called reduced if it has no non-zero
nilpotent elements.

\begin{lemma} [{\cite[Theorem 2.2]{G2}}] \label{B5} Let
$R$ be a ring. Then $\wgdim(R)\leq 1$ if and only if $R$ is a
Gaussian reduced ring.\end{lemma}

\begin{proof}
Assume that $\wgdim(R)\leq 1$. By {\cite [Corollary 4.2.6]{G1}},
$R_{p}$ is a valuation domain for every prime ideal $p$ of $R$. As
valuation domains are Gaussian, $R$ is locally Gaussian, and
therefore Gaussian. Further, $R$ is reduced. For, let $x\in R$
such that $x$ is nilpotent. We claim that $x=0$. Suppose not and let
$n\geq 2$ be an integer such that $x^{n}=0$. Then there exists a
prime ideal $q$ in $R$ such that $x\neq 0$  in $R_{q}$
{\cite[Proposition 3.8]{AM}}. It follows that $x^{n}=0$ in
$R_{q}$, a contradiction since $R_{q}$ is a domain.

Conversely, since $R$ is Gaussian reduced, $R_{p}$ is a local,
reduced, Gaussian ring for any prime ideal $p$ of $R$. We claim
that $R_{p}$ is a domain. Indeed, let $a$ and $b$ in $R_{p}$ such
that $ab=0$. By Lemma \ref{B4}, $(a, b)^{2}$=$(b)^{2}$ or
$(a^{2})$. Say $(a, b)^{2}=(b^{2})$. Then $a^{2}=tb^{2}$ for some
$t\in R_{p}$. Thus $a^{3}=tb(ab)=0$. Since $R_{p}$ is reduced,
$a=0$, and $R_{p}$ is a domain. Therefore $R_{p}$ is a valuation
domain for all prime ideals $p$ of $R$. So $\wgdim(R)\leq 1$ by
{\cite [Corollary 4.2.6]{G1}}.
\end{proof}

\begin{lemma} [{\cite[Theorem 3.3]{BG2}}] \label{B61} Let $R$ be a Pr\"ufer ring. Then $R$ is
Gaussian if and only if $Q(R)$ is Gaussian. \hfill{$\square$}
\end{lemma}

\begin{lemma} [{\cite[Theorem 3.12(ii)]{BG2}}] \label{B6} Let
$R$ be a ring. Then $\wgdim(R)\leq 1$ if and only if $R$ is a
Pr\"ufer ring and  $\wgdim(Q(R))\leq 1$.\end{lemma}

\begin{proof}
If $\wgdim(R)\leq 1$, $R$ is Pr\"ufer and, by localization,
$\wgdim(Q(R))\leq 1$. Conversely, assume that $R$ is a Pr\"ufer
ring such that $\wgdim(Q(R))\leq 1$. By Lemma \ref{B5}, $Q(R)$ is
a Gaussian reduced ring. So $R$ is reduced and, by Lemma
\ref{B61}, $R$ is Gaussian. By Lemma \ref{B5}, $\wgdim(R)\leq 1$.
\end{proof}

\textbf{Proof of Proposition \ref{B2}.}
Assume that $\wgdim(R) = n < \infty$ and let  $I$ be any finitely
generated ideal of $R$. Then $I$  has a finite weak dimension.
Since $R$ is a coherent ring,  $I$ is finitely presented. Hence
the weak dimension of I equals its projective dimension by  {\cite
[Corollary 2.5.5]{G1}}. Whence, as $I$ is an arbitrary finitely
generated ideal of $R$, $R$ is a regular ring. So, by {\cite
[Corollary 6.2.4]{G1}}, $Q(R)$ is von Neumann regular. By Lemma
\ref{B6}, $\wgdim(R)\leq 1$. $\Box$
\bigskip

The following is an example of a coherent Pr\"ufer ring with
infinite  weak global dimension.

\begin{example} \rm
Let $R=\mathbb{R} \ltimes \mathbb{C}$. Then $R$ is coherent by
{\cite[Theorem 2.6]{KM}}, Pr\"ufer by Theorem \ref{T3}, and
$\wgdim(R)= \infty$ by Lemma \ref{T2}.
\end{example}

In order to study the weak global dimension of an arbitrary
Gaussian ring, we make the following reductions:

(1) We may assume that $R$ is a local Gaussian ring since
$\wgdim(R)$ is the supremum of $\wgdim(R_{m})$ for all  maximal
ideal $m$ of $R$\ {\cite [Theorem 1.3.14 (1)]{G1}}.

(2) We may assume that $R$ is a non-reduced local Gaussian ring
since every reduced Gaussian ring has weak global dimension at
most $1$ by Lemma \ref{B5}.

(3) Finally, we may assume that $(R, \m)$ is a local Gaussian ring
with the maximal ideal $\m$ such that $\m = \Nil(R)$. For, the
prime ideals of a local Gaussian ring $R$ are linearly ordered, so
that $\Nil(R)$ is a prime ideal, and $\wgdim(R) \geq
\wgdim(R_{\Nil(R)})$.

Next we announce  the second  main result of this section.

\begin{thm} [{\cite[Theorem 6.4]{BG2}}] \label{B7}  Let $R$ be a Gaussian ring with a maximal ideal $\m$ such that $\Nil(R_{\m})$ is a nonzero nilpotent ideal. Then $\wgdim(R) = \infty$.
\end{thm}

The proof of this theorem involves  the following results:

\begin{lemma} \label{B8} Consider the following exact sequence of
$R$-modules
\begin{center} $0\longrightarrow M^{\prime} \longrightarrow M
\longrightarrow M^{\prime\prime} \longrightarrow 0$
\end{center}
where  $M$ is flat. Then either the three modules are flat or
$\wdim(M^{\prime\prime})=\wdim(M^{\prime})\\ +1$.
\end{lemma}

\begin{proof}
This is a classic result. We offer here a proof for the sake of completeness. Suppose that $M^{\prime\prime}$ is flat. Then by the long exact
sequence theorem \cite [Theorem 8.3]{Ro} we get the exact sequence
$$0=Tor_{2}(M^{\prime\prime},N)\longrightarrow\
Tor_{1}(M^{\prime},N) \longrightarrow\  Tor_{1}(M,N)=0$$  for any
$R$-module $N$. Hence $Tor_{1}(M^{\prime},N)=0$ which implies that
$M^{\prime}$ is flat.

Next, assume that $M^{\prime\prime}$ is not flat. In this case, we
claim that $$\wdim(M^{\prime\prime})=\wdim(M^{\prime})+1.$$ Indeed,
let $\wdim(M^{\prime})=n$. Then we have the exact sequence
\begin{center} $0=Tor_{n+2}(M,N)\longrightarrow
Tor_{n+2}(M^{\prime\prime},N) \longrightarrow
Tor_{n+1}(M^{\prime},N)=0$ \end{center} for any $R$-module $N$.
Hence $Tor_{n+2}(M^{\prime\prime},N)=0$ for any $R$-module $N$
which implies
\begin{center}$\wdim(M^{\prime\prime})\leq n+1 = \wdim(M^{\prime})+1$ \end{center}
Now let $\wdim(M^{\prime\prime})=m$. Then we have the exact
sequence $$0=Tor_{m+1}(M^{\prime\prime},N)\longrightarrow
Tor_{m}(M^{\prime},N) \longrightarrow Tor_{m}(M,N)=0$$ for any
$R$-module $N$. Hence $Tor_{m}(M^{\prime},N)=0$ for any $R$-module
$N$ which implies that
\begin{center}$\wdim(M^{\prime\prime})=m \geq \wdim(M^{\prime})+1$
\end{center}
Consequently,  $\wdim(M^{\prime\prime})=\wdim(M^{\prime})+1$.
\end{proof}

Recall that an exact sequence of R-modules $$0 \longrightarrow
M^{\prime} \longrightarrow M \longrightarrow M^{\prime\prime}
\longrightarrow 0$$ is pure if it remains exact when tensoring it
with any $R$-module. In this case, we say that $M^{\prime}$ is a
pure submodule of $M$ \cite{Ro}.
\begin{lemma} [{\cite[Lemma 6.2]{BG2}}] \label{B9} Let $(R, \m)$ be a local ring which is not a field.
Then $\wdim(R/\m)=\wdim( \m)+1$. \end{lemma}

\begin{proof}
 Consider the short exact sequence \begin{center} $0\rightarrow \m
\rightarrow R \rightarrow R/\m \rightarrow 0$. \end{center} Assume
that $R/\m$ is flat. By {\cite [Theorem 1.2.15 (1,2,3)]{G1}}, $\m$
is pure and $(aR)\m=aR\cap \m=aR$ for all $a\in \m$. Hence
$a\m=aR$, for all $a\in \m$, and so by Nakayama's Lemma,
$a=0$, absurd. By Lemma \ref{B8}, $\wdim(R/ \m)=\wdim_{R}( \m)+1$.
\end{proof}

\begin{proposition}[{\cite[Proposition 6.3]{BG2}}] \label{B10} Let $(R, \m)$ be a local ring with nonzero
nilpotent maximal ideal. Then $\wdim(\m)=\infty$.
\end{proposition}

\begin{proof}
Let $n$ be the minimum integer such that $\m^{n}=0$. We claim that
for all $1\leq k <n$, $\wdim(\m^{n-k})=\wdim( \m)+1$. Indeed, let
$k=1$. Then $\m^{n-1}\m=0$, so $\m^{n-1}$ is an $(R/\m)$-vector
space, hence $0\neq \m^{n-1}\cong \bigoplus R/\m$, implies that
$\wdim_{R}(\m^{n-1})= \wdim(R/\m)= \wdim( \m)+1$ by Lemma \ref{B9}
. Now let $h$ be the maximum integer in $\{1, ..., n-1\}$ such
that $\wdim(\m^{n-k})=\wdim( \m)+1$ for all $k\leq h$. Assume by
way of contradiction that $h<n-1$. Then we have the exact
sequence:
\begin{equation*}
0\rightarrow \m^{n-h} \rightarrow \m^{n-(h+1)}\rightarrow \m^{n-(h+1)}/\m^{n-h} \rightarrow 0\hspace{1cm}  (\ast)
\end{equation*}
where  $\m^{n-(h+1)}/\m^{n-h}$ is a nonzero $(R/\m)$-vector space.
So by Lemma \ref{B9}, we have $\wdim(\m^{n-(h+1)}/ \m^{n-h})= \wdim(
\m)+1$. By hypothesis, $\wdim(\m^{n-h})=\wdim(\m) +1$. Let us
show that $\wdim(\m^{n-(h+1)})=\wdim( \m)+1$. Indeed, if
$l:=\wdim( \m) +1$, then by applying the long exact sequence
theorem to $(\ast)$, we get
\begin{center} $0=Tor_{l+1}(\m^{n-h},N)\longrightarrow
Tor_{l+1}(\m^{n-(h+1)},N) \longrightarrow
Tor_{l+1}(\frac{\m^{n-(h+1)}}{\m^{n-h}},N)=0$
\end{center} for any $R$-module $N$.
Hence $Tor_{l+1}(\m^{n-(h+1)},N)=0$ for any $R$-module $N$ which
implies
\begin{center}$\wdim(\m^{n-(h+1)})\leq l = \wdim(\m)+1$ \end{center}

Further, if $\wdim(\m^{n-(h+1)})\lneqq l$, then we have
\begin{center} $0=Tor_{l+1}(\frac{\m^{n-(h+1)}}{\m^{n-h}},N)\longrightarrow Tor_{l}(\m^{n-h},N)\longrightarrow
Tor_{l}(\m^{n-(h+1)},N)=0$
\end{center} for any $R$-module $N$.
Hence $Tor_{l}(\m^{n-h},N)=0$ for any $R$-module $N$ which implies
that $\wdim(\m^{n-h)})\leq l-1$, absurd. Hence
$\wdim(\m^{n-(h+1)})=\wdim( \m)+1$, the desired contradiction.
Therefore the claim is true and, in particular, for $k=n-1$, we
have $\wdim(\m)=\wdim( \m)+1$, which yields $\wdim(\m)=\infty$.
\end{proof}

\textbf{Proof of Theorem \ref{B7}.}
Suppose that $R$ is Gaussian and $\m$ is a maximal ideal in $R$
such that  $\Nil(R_{\m})$ is a nonzero nilpotent ideal. Then
$R_{\m}$ is also Gaussian and $\Nil(R_{\m})$ is a prime ideal in
$R$. Moreover $\Nil(R_{\m})=pR_{\m}\neq 0$ for some prime ideal
$p$ in $R$. Now, the maximal ideal $pR_{p}$ of $R_{p}$ is nonzero
since $0 \neq pR_{\m} \subseteq pR_{p}$. Also by assumption, there
is a positive integer $n$ such that $(pR_{\m})^{n}=0$, whence
$p^{n}=0$. So $(pR_{p})^{n}=0$ and hence $pR_{p}$ is nilpotent.
Therefore $R_{p}$ is a local ring with nonzero nilpotent maximal
ideal. By Proposition \ref{B10}, $\wgdim(R_{p})=\infty$. Since
$\wgdim(R) \geq \wgdim(R_{S})$ for any localization $R_{S}$ of
$R$, we get $\wgdim(R)=\infty$. $\Box$
\bigskip

In the previous section, we saw that the weak global dimension of an arithmetical  ring is $0$, $1$, or $\infty$. In this section, we saw that the same result holds if $R$ is Pr\"ufer coherent or
$R$ is a Gaussian ring with a maximal ideal $\m$ such that $\Nil(R_{\m})$ is a nonzero nilpotent ideal.

The question of whether this result is true for an arbitrary Gaussian ring was the object of Bazzoni-Glaz conjecture which sustained that the weak global dimension of a
Gaussian ring is $0$, $1$, or $\infty$. In a first preprint \cite{DT1}, Donadze and Thomas claimed to prove this conjecture in all cases except when the ring $R$ is a non-reduced local Gaussian ring with nilradical $N$ satisfying $N^{2} = 0$. Then in a second preprint \cite{DT2}, they claimed to prove the conjecture for all cases.

\end{section}

\begin{section}{Gaussian rings via trivial ring extensions}\label{T}

\noindent In this section, we will use trivial ring extensions to construct
new examples of non- arithmetical Gaussian rings , non-Gaussian
Pr\"ufer  rings, and illustrative examples for  Theorem \ref{A4}
and Theorem \ref{B7}.
Let $A$ be a ring and $M$ an $R$-module. The trivial ring
extension of $A$ by $M$ (also called the idealization of $M$ over
$A$) is the ring $R:=A \ltimes M$ whose underlying group is
$A\times M$ with multiplication given by
$$(a,x)(a^{\prime},x^{\prime})=(aa^{\prime},ax^{\prime}+a^{\prime}x).$$

Recall that if $I$ is an ideal of $A$ and $M'$ is a submodule of
$M$ such that $IM\subseteq M'$, then $J:=I \ltimes M'$ is an ideal
of $R$; ideals of $R$ need not be of this form {\cite[Example
2.5]{KM}}. However, the form of the prime (resp., maximal) ideals
of $R$ is $p \ltimes M$, where $p$ is a prime (resp., maximal)
ideal of $A$ {\cite[Theorem 25.1(3)]{H}}. Suitable background on
trivial extensions is {\cite{G1, H, KM}}.

The following lemma is useful for  the construction of rings with
infinite  weak global dimension.
\begin{lemma} [{\cite[Lemma 2.3]{BKM}}] \label{T2} Let $K$ be a field, $E$ a nonzero
$K$-vector space, and $R:=K \ltimes E$. Then $\wgdim(R)=\infty$.
\end{lemma}

\begin{proof}
First note that $R^{(I)}\cong A^{(I)} \ltimes E^{(I)}$. So let us
identify $R^{(I)}$ with $A^{(I)} \ltimes E^{(I)}$ as $R$-modules.
Now let $\{f_{i}\}_{i\in I}$ be a basis of $E$ and $J:=0 \ltimes
E$. Consider the $R$-map $u:R^{(I)}\longrightarrow J$ defined by
$\displaystyle u((a_{i},e_{i})_{i\in I})=(0,\sum \limits_{i\in
I}a_{i}f_{i})$. Then we have the following short exact sequence of
$R$-modules
\begin{center} $0\longrightarrow \Ker(u) \longrightarrow R^{(I)}\overset{u}
{\longrightarrow} J \longrightarrow 0$  \end{center} But $\Ker(u)=
0 \ltimes E^{(I)}$. Indeed, clearly $0 \ltimes E^{(I)} \subseteq
\Ker(u)$. Now suppose $u((a_{i},e_{i}))=(0,0)$. Then $\sum
\limits_{i\in I}a_{i}f_{i}=0$, hence $a_{i}=0$ for each $i$ as
$\{f_{i}\}_{i\in I}$ is a basis for $E$ and we have the equality.
Therefore the above exact sequence becomes
\begin{center} $0\longrightarrow 0\ltimes E^{(I)}\longrightarrow R^{(I)}\overset{u}
{\longrightarrow} J \longrightarrow 0$\hspace{1cm} ($\ast$) \end{center} We
claim that $J$ is not flat. Suppose not. Then by {\cite[Theorem 3.55]
{Ro}}, $0 \ltimes E^{(I)}\bigcap JR^{(I)}= (0 \ltimes E^{(I)})J$.
But $(0 \ltimes E^{(I)})J=0$. We use the above identification
to obtain $0= 0 \ltimes E^{(I)}\bigcap JR^{(I)}=(J)^{(I)}\bigcap
J^{(I)} = J^{(I)}= 0 \ltimes E^{(I)}$, absurd (since $E\neq 0$).

Now, by Lemma \ref {B8}, $\wdim(J)= \wdim(J^{(I)})+1=\wdim(J)+1$.
It follows that $\wgdim(R)= \wdim(J)= \infty$.
\end{proof}

Next, we announce the main result of this section.
\begin{thm} [{\cite[Theorem 3.1]{BKM}}]\label{T3}
Let $(A, \m)$ be a local ring, $E$ a nonzero
$\frac{A}{\m}$-vector space, and $R:=A\ltimes E$  the trivial ring
extension of $A$ by $E$. Then:
\begin{enumerate}
\item $R$ is a total ring of quotients and hence a Pr\"ufer ring.

\item $R$ is Gaussian if and only if $A$ is Gaussian.

\item $R$ is arithmetical if and only if $A:=K$ is a field and
$\dim_{K}(E)=1$.

\item  $\wgdim(R)\gneqq 1$. If $\m$ admits a minimal generating set,
then $\wgdim(R)$ is infinite. \end{enumerate}
\end{thm}

\begin{proof}
(1) Let $(a,e)\in R$. Then either $a\in \m$ in which
case we get $(a,e)(0,e)=(0,ae)=(0,0)$; or $a \notin \m$ which implies $a$
is a unit and hence $(a,e)(a^{-1},-a^{-2}e)=(1,0)$, the unity of
$R$. Therefore $R$ is a total ring of quotients and hence a Pr\"ufer
ring.

(2) Suppose that $R$ is Gaussian. Then, since $A\cong
\frac{R}{0\ltimes E}$ and the Gaussian property is stable under
factor rings, $A$ is Gaussian.

Conversely, assume that $A$ is Gaussian and let $F:=\sum
(a_{i},e_{i})X^{i}$ be a polynomial in $R[X]$. Then if $a_{i}
\notin \m$ for some $i$ and in this case $(a_{i},e_{i})$ is
invertible since we have $(a_{i},e_{i})(a^{-1}_{i},-a^{-2}e_{i})=(1,0)$. We
claim that $F$ is Gaussian. Indeed, for any  $G\in R[X]$,  we have
$c(F)c(G)=Rc(G)=c(G)\subseteq c(FG)$. The reverse inclusion always
holds. If $a_{i} \in \m$ for each $i$, let $G:=\sum
(a_{j}^{\prime},e_{j}^{\prime})X^{j}\in R[X]$. We may assume,
without loss of generality, that $a_{j}^{\prime} \in \m$\ for each
$j$ (otherwise, we return to the first case) and let $f:=\sum
a_{i}X^{i}$ and $g:=\sum a_{j}^{\prime}X^{j}$ in $A[X]$. Then
$c(FG)=c(fg)\ltimes c(fg)E$. But since $E$ is an
$\frac{A}{\m}$-vector space, $\m E=0$ yields $c(FG)=c(fg)\ltimes
0= c(f)c(g)\ltimes 0=c(F)c(G)$, since $A$ is Gaussian. Therefore
$R$ is Gaussian, as desired.

(3) Suppose that $R$ is arithmetical. First we claim that $A$ is
a field. On the contrary, assume that $A$ is not a field. Then $\m
\neq0$, so there is $ a\neq 0 \in \m$. Let $e\neq 0 \in E$. Since
$R$ is a local arithmetical ring (i.e., chained ring), either
$(a,0)=(a^{\prime},e^{\prime})(0,e)=(0,a^{\prime}e)$ for some
$(a^{\prime},e^{\prime})\in R$ which contradicts $a\neq 0$; or
$(0,e)=(a^{\prime\prime},e^{\prime\prime})(a,0)=(a^{\prime} a,0)$
for some $(a^{\prime\prime},e^{\prime\prime})\in R$ which
contradicts $e\neq 0$. Hence  $A$ is a field. Next, we show that
$\dim_{K}(E)=1$. Let $e,e^{\prime}$ be two nonzero vectors in $E$.
We claim that they are linearly dependent. Indeed, since $R$ is a
local  arithmetical ring, either
$(0,e)=(a,e^{\prime\prime})(0,e^{\prime})=(0,ae^{\prime})$ for
some $(a,e^{\prime\prime})\in R$, hence $e=ae^{\prime}$;  or
similarly if $(0,e^{\prime})\in (0,e)R$. Consequently,
$\dim_{K}(E)=1$.

Conversely, let $J$ be a nonzero ideal in $K\ltimes K$ and let
$(a,b)$ be a nonzero element of $J$. So $(0,a^{-1})(a,b)=(0,1)\in J$. Hence
$0\ltimes K \subseteq J$. But $0\ltimes K$ is maximal since $0$ is
the maximal ideal in $K$. So the ideals of $K\ltimes K$ are
$(0,0)K\ltimes K$, $0\ltimes K=R(0,1)$, and $K\ltimes K$.
Therefore $K\ltimes K$ is a principal ring and hence arithmetical.

(4) First $\wgdim(R)\gneqq 1$. Let $J:=0\ltimes E$ and
$\{f_{i}\}_{i\in I}$ be a basis of the $\frac{A}{\m}$-vector space
$E$. Consider the map $u:R^{(I)}\longrightarrow J$ defined by
$u((a_{i},e_{i})_{i\in I})=(0,\sum \limits_{i\in I} a_{i}f_{i})$.
Here we are using the same identification that has been used in
Lemma \ref {T2}. Then clearly $\Ker(u)=(\m \ltimes E)^{(I)}$.
Hence we have the short exact sequence of $R$-modules

\begin{center} $0\longrightarrow (\m \ltimes E)^{(I)}
\longrightarrow R^{(I)}\overset{u} {\longrightarrow} J
\longrightarrow 0$ \hspace{1cm} ($1$) \end{center}

We claim that $J$ is not flat. Otherwise, by \cite
[Theorem3.55]{Ro}, we have  $$J^{(I)}=(\m \ltimes E)^{(I)}\cap
JR^{(I)}=J(\m \ltimes E^{(I)})=0.$$ Hence, by \cite [Theorem
2.44]{Ro}, $\wgdim(R)\gneqq 1$.

Next, assume that $\m$ admits a minimal generating set. Then $\m
\ltimes E$ admits a minimal generating set (since $E$ is a vector
space). Now let $(b_{i},g_{i})_{i\in L}$ be a minimal generating
set of $\m \ltimes E$. Consider the $R$-map
$v:R^{(L)}\longrightarrow \m \ltimes E$ defined by
$v((a_{i},e_{i})_{i\in L})=\sum \limits_{i\in
L}(a_{i},e_{i})(b_{i},g_{i})$. Then we have the exact sequence

\begin{center} $0\longrightarrow \Ker(v)
\longrightarrow R^{(L)}\overset{v} {\longrightarrow} \m \ltimes E
\longrightarrow 0$ \hspace{1cm} ($2$) \end{center} We claim that
$\Ker(v)\subseteq (\m \ltimes E)^{(L)}$. On the contrary, suppose
that there is $x=((a_{i},e_{i})_{i\in L})\in \Ker(v)$ and $x
\notin (\m \ltimes E)^{(L)}$. Then $\sum \limits_{i\in L}
(a_{i},e_{i})(b_{i},g_{i})=0$ and as $x \notin (\m \ltimes
E)^{(L)}$, there is $(a_{j},e_{j})$ with $a_{j} \notin \m$. So
that $(a_{j},e_{j})$ is a unit, which contradicts the minimality
of $(b_{i},g_{i})_{i\in L}$. It follows that $$\Ker(v)=V\ltimes
E^{(L)}=(V\ltimes 0)\bigoplus(0\ltimes E^{(L)})=(V\ltimes
0)\bigoplus J^{(L)}$$ where $\displaystyle V:=\{(a_{i})_{i\in
L}\in \m^{i}\ |\ \sum \limits_{i\in L}a_{i}b_{i}=0\}$. Indeed, if
$x\in \Ker(v)$, then $x=(a_{i},b_{i})_{i\in L}$ where $a_{i}\in
\m$, $b_{i}\in E$, with  $\displaystyle \sum \limits_{i\in
L}a_{i}b_{i}=0$, hence $\Ker(v)\subseteq V\ltimes E^{(L)}$. The
other inclusion is trivial. Now, by Lemma \ref{B8} applied to
($1$), we get $$\wdim(J)=\wdim((\m \ltimes E)^{I})+1=\wdim(\m
\ltimes E)+1.$$ On the other hand, from ($2$) we obtain
$$\wdim(J)\leq \wdim(V \ltimes 0 \oplus J^{L})=\wdim(\Ker(v)) \leq
\wdim(\m \ltimes E).$$ It follows that $$\wdim(J)\leq
\wdim(J)-1.$$ Consequently, $\wgdim(R)=\wdim(J)=\infty.$
\end{proof}

Next, we give examples of non-arithmetical Gaussian rings.

\begin{example} \rm \label{T4}
\begin{enumerate}
\item Let $p$ be a prime number. Then ($\mathbb{Z}_{(p)}, p\mathbb{Z}_{(p)}$) is a
non-trivial valuation domain. Hence $\mathbb{Z}_{(p)} \ltimes
\frac{\mathbb{Z}}{p\mathbb{Z}}$ is a non-arithmetical Gaussian
total ring of quotients by Theorem \ref{T3}.

\item Since $\dim_{\mathbb{R}}(\mathbb{C})=2\gneqq 1$, $\mathbb{R}
\ltimes \mathbb{C}$ is a non arithmetical Gaussian total ring of
quotient. In general, if $K$ is a field and $E$ is a $K$-vector
space with $\dim_{K}(E)\gneqq 1$, then $R:=K\ltimes E$ is a
non-arithmetical Gaussian total ring of quotients by Theorem
\ref{T3}.
\end{enumerate}
\end{example}

Next, we provide examples of  non-Gaussian total rings of
quotients and hence non-Gaussian Pr\"ufer rings.
\begin{example} \rm
Let $(A,\m)$ be a non-valuation  local domain. By Theorem
\ref{T3}, $R:=A\ltimes \frac{A}{\m}$ is a non-Gaussian total ring
of quotients, hence a non-Gaussian Pr\"ufer ring.
\end{example}
The following is an illustrative example for Theorem \ref {A4}.

\begin{example} \rm
Let $R:=\mathbb{R} \ltimes \mathbb{R}$. Then  $R$ is  a local ring
with maximal ideal $0 \ltimes \mathbb{R}$ and $\Ze(R)=0 \ltimes
\mathbb{R}$. Further, $R$ is arithmetical by Theorem \ref {T3}. By
Osofsky's Theorem (Theorem \ref{A4}) or by  Lemma \ref{T2},
$\wgdim(R)=\infty$.
\end{example}

Now we give an example of a non-coherent local Gaussian ring with
nilpotent maximal ideal and infinite weak global dimension (i.e.,
an illustrative example for Theorem \ref{B7}).
\begin{example} \rm

Let $K$ be a field and  $X$  an indeterminate over $K$ and let
$R:=K\ltimes K[X]$. Then:
\begin{enumerate}
\item $R$ is a non-arithmetical Gaussian ring since $K$ is Gaussian and $\dim_{K} (K[X])=\infty$ by Theorem
\ref{T3}.
\item $R$ is not a coherent ring since  $\dim_{K} (K[X])=\infty$ by \cite[Theorem
2.6]{KM}.
\item $R$ is local with maximal ideal $\m =0 \ltimes K[X]$
by \cite[Theorem 25.1(3)]{H}. Also $\m$ is nilpotent since
$\m^{2}=0$. Therefore, by Theorem \ref{B7}, $\wgdim(R)= \infty$.

\end{enumerate}
\end{example}

\end{section}

\begin{section}{Weak global dimension of fqp-rings}\label{Fq}

\noindent Recently, Abuhlail, Jarrar, and Kabbaj studied  commutative rings
in which every finitely generated ideal is quasi-projective
(fqp-rings). They investigated the correlation of fqp-rings with
well-known Pr\"ufer conditions; namely, they  proved that
fqp-rings  stand strictly between the two classes of arithmetical
rings and Gaussian rings {\cite[Theorem 3.2]{AJK}}. Also they
generalized Osofsky's Theorem on the weak global dimension of
arithmetical rings (and partially resolved Bazzoni-Glaz's related
conjecture on Gaussian rings) by proving that the weak global
dimension of an fqp-ring  is $0$, $1$, or $\infty$ {\cite[Theorem
3.11]{AJK}}. In this section, we will give the proofs of the above
mentioned results. Here too, the needed examples in this section
will be constructed by using trivial ring extensions. We start by
recalling some definitions.
\begin{definition} \label {F1} \rm
\begin{enumerate}
\item Let $M$ be an $R$-module. An $R$-module $M^{\prime}$ is
$M$-projective if the map $\psi
:\Hom_{R}(M^{\prime},M)\longrightarrow
\Hom_{R}(M^{\prime},\frac{M}{N})$ is surjective for every
submodule $N$ of $M$.
\item $M^{\prime}$ is quasi-projective if it is
$M^{\prime}$-projective.
\end{enumerate}
\end{definition}

\begin{definition} \label {F2} \rm
A commutative ring $R$ is said to be an fqp-ring if every finitely
generated ideal of $R$ is quasi-projective.
\end{definition}

The following theorem establishes the relation between the class
of fqp-rings and the two classes of arithmetical and Gaussian
rings.

\begin{thm} [{\cite[Theorem 3.2]{AJK}}]\label {F3}
For a ring $R$, we have
$$ R\ arithmetical\  \Rightarrow\ R\ fqp-ring\ \Rightarrow\ R\
Gaussian$$

where the implications are irreversible in general.
\end{thm}

The proof of this theorem needs  the following results.

\begin{lemma} [{\cite[Lemma 2.2]{AJK}}]\label{F4}
Let $R$ be a ring and let $M$ be a finitely generated $R$-module. Then
$M$ is quasi-projective if and only if $M$ is projective over
$\frac{R}{\Ann(M)}$. $\Box$
\end{lemma}

\begin{lemma}[{\cite[Corollary 1.2]{FH}}] \label{F41} Let ${M_{i}}_{1\leq i \leq n}$ be a family of $R$-modules.
Then:\\ $\bigoplus_{i=1}^{n}M_{i}$ is quasi-projective if and only if
$M_{i}$ is $M_{j}$-projective $\forall$ $i,j\in \{1,\ 2,\ ...,\ \}$.
$\Box$

\end{lemma}
\begin{lemma} [{\cite[Lemma 3.6]{AJK}}] \label{F5}
Let $R$ be an fqp-ring. Then $S^{-1}R$ is an fqp-ring, for any
multiplicative closed subsets of $R$.
\end{lemma}

\begin{proof}
Let $J$ be a finitely generated ideal of $S^{-1}R$. Then
$J=S^{-1}I$ for some finitely generated ideal $I$ of $R$. Since
$R$ is an fqp-ring, $I$ is quasi-projective and hence, by Lemma
\ref{F4}, $I$ is projective over $\frac{R}{\Ann(I)}$. By
{\cite[Theorem 3.76]{Ro}}, $J:=S^{-1}I$ is projective over
$\frac{S^{-1}R}{S^{-1}\Ann(I)}$. But
$S^{-1}\Ann(I)=\Ann(S^{-1}I)=\Ann(J)$ by {\cite[Proposition
3.14]{AM}}. Therefore $J:=S^{-1}I$ is projective over
$\frac{S^{-1}R}{\Ann(S^{-1}I)}$. Again by Lemma \ref{F4}, $J$ is
quasi-projective. It follows that $S^{-1}R$ is an fqp-ring.
\end{proof}

\begin{lemma} [{\cite[Lemma 3.8]{AJK}}] \label{F6} Let $R$ be a local ring and $a,\ b$
two nonzero elements of $R$ such that $(a)$ and $(b)$ are
incomparable. If $(a,\ b)$ is quasi-projective, then $(a)\cap
(b)=0$, $a^{2}=b^{2}=ab=0$, and $\Ann(a)=\Ann(b)$.
\end{lemma}

\begin{proof}
Let $I:=(a,\ b)$ be quasi-projective. Then by {\cite[Lemma
2]{Tu}}, there exist $f_{1},\ f_{2} \in \End_{R}(I)$ such that
$f_{1}(I)\subseteq (a)$, $f_{2}(I)\subseteq (b)$, and
$f_{1}+f_{2}=1_{I}$. Now let $x\in (a)\cap (b)$. Then
$x=r_{1}a=r_{2}b$ for some $r_{1},\ r_{2}\ \in R$. But
$x=f_{1}(x)+f_{2}(x)=f_{1}(r_{1}a)+f_{2}(r_{2}b)=r_{1}f_{1}(a)+r_{2}f_{2}(b)=r_{1}a^{\prime}a+r_{2}b^{\prime}b=a^{\prime}x+b^{\prime}x$
where $a^{\prime},\ b^{\prime}\ \in R$. We claim that $a^{\prime}$
is a unit. Suppose not. Since $R$ is  local, $1-a^{\prime}$ is a unit.
But $a=f_{1}(a)+f_{2}(a)=a^{\prime}a+f_{2}(a)$. Hence
$(1-a^{\prime})a=f_{2}(a)\subseteq (b)$ which implies that $a\in
(b)$. This is absurd since $(a)$ and $(b)$ are incomparable. Similarly,
$b^{\prime}$ is a unit. It follows that
$(a^{\prime}-(1-b^{\prime}))$ is a unit. But
$x=a^{\prime}x+b^{\prime}x$ yields
$(a^{\prime}-(1-b^{\prime}))x=0$. Therefore $x=0$ and $(a)\cap
(b)=0$.

Next, we prove that $a^{2}=b^{2}=ab=0$. Obviously,  $(a)\cap
(b)=0$ implies that  $ab=0$. So it remains to prove that
$a^{2}=b^{2}=0$. Since $(a)\cap (b)=0$, $I=(a)\oplus (b)$. By
Lemma \ref{F41}, $(b)$ is $(a)$-projective. Let $\varphi:
(a)\longrightarrow \frac{(a)}{a\Ann(b)}$ be the canonical map and
$g: (b)\longrightarrow \frac{(a)}{a\Ann(b)}$ be defined by
$g(rb)=r\bar{a}$.  If $r_{1}b=r_{2}b$, then $(r_{1}-r_{2})b=0$.
Hence $r_{1}-r_{2} \in \Ann(b)$ which implies that
$(r_{1}-r_{2})\bar{a}=0$. So $g(r_{1}b)=g(r_{2}b)$. Consequently,
$g$ is well defined. Clearly $g$ is an $R$-map. Now, since $(b)$
is $(a)$-projective, there exists an $R$-map $f:
(b)\longrightarrow (a)$ with $\varphi \circ f=g$. For $b$, we have
$f(b)\in (a)$, hence $f(b)=ra$ for some $r\in R$. Also $(\varphi
\circ f)(b)=g(b)$. Hence $f(b)-a\in a\Ann(b)$. Whence  $ra-a=at$
for some  $t\in \Ann(b)$ which implies that $(t+1)a=ra$. By
multiplying the last equality by $a$ we obtain,
$(t+1)a^{2}=ra^{2}$. But $ab=0$ implies $0=f(ab)=af(b)=ra^{2}$.
Hence $(t+1)a^{2}=0$. Since $t\in \Ann(b)$ and $R$ is local,
$(t+1)$ is a unit. It follows that $a^{2}=0$. Likewise $b^{2}=0$.

Last, let $x\in \Ann(b)$. Then $f(xb)=xra=0$. The above equality
$(t+1)a=ra$ implies $(t+1-r)a=0$. But $t+1$ is a unit and $R$ is
local. So that $r$ is a unit ($b\neq 0$). Hence $xa=0$. Whence
$x\in \Ann(a)$ and $\Ann(b)\ \subseteq\ \Ann(a)$. Similarly we can
show that $\Ann(a)\ \subseteq\ Ann(b)$. Therefore $\Ann(a)\ =\
\Ann(b)$.
\end{proof}

\textbf{Proof of Theorem \ref{F3}.}
$R$ arithmetical $\Rightarrow$ $R$ fqp-ring.

Let $R$ be an arithmetical ring, $I$ a nonzero finitely generated
ideal of $R$, and $p$ a prime ideal of $R$. Then $I_{p}:=IR_{p}$
is finitely generated. But $R$ is arithmetical, hence $R_{p}$ is a
chained ring and $I_{p}$ is a principal ideal of $R_{p}$. By
{\cite{Ko}}, $I_{p}$ is quasi-projective. By {\cite[19.2]{W1}} and
{\cite{W2}}, it suffices to prove that $(\Hom_{R}(I,\ I))_{p}\cong
\Hom_{R_{p}}(I_{p},\ I_{p})$. But $ \Hom_{R_{p}}(I_{p},\
I_{p})\cong \Hom_{R}(I,\ I_{p})$ by the adjoint isomorphisms
theorem {\cite[Theorem 2.11]{Ro}} (since $
\Hom_{S^{-1}R}(S^{-1}N,S^{-1}M)\cong \Hom(N,S^{-1}M)$ where
$S^{-1}N\cong N\bigotimes _{R}S^{-1}R$ and $S^{-1}M \cong
\Hom_{S^{-1}R}(S^{-1}R,S^{-1}M)$). So let us prove that
$$(\Hom_{R}(I,\ I))_{p}\cong \Hom_{R}(I,\ I_{p}).$$ Let $$\phi
:(\Hom_{R}(I,\ I))_{p} \longrightarrow \Hom_{R}(I,\ I_{p})$$ be
the function defined by  $\frac{f}{s}\in (\Hom_{R}(I,\ I))_{p}$,
$\phi(\frac{f}{s}):I\longrightarrow I_{p}$ with
$\phi(\frac{f}{s})(x)=\frac{f(x)}{s}$, for each  $x\in I$. Clearly
$\phi$ is a well-defined  $R$-map. Now suppose that
$\phi(\frac{f}{s})=0$. $I$ is finitely generated, so let
$I=(x_{1},\ x_{2}, ...,\ x_{n})$, where $n$ is an integer. Then
for every $i\in \{1,\ 2,\ ...,\ n\}$,
$\phi(\frac{f}{s})(x_{i})=\frac{f(x_{i})}{s}=0$, whence there
exists $t_{i}\in R\setminus p$ such that $t_{i}f(x_{i})=0$. Let
$t:=t_{1}t_{2}...t_{n}$. Clearly, $t\in R\setminus p$ and
$tf(x)=0$, for all $x\in I$. Hence $\frac{f}{s}=0$. Consequently,
$\phi$ is injective. Next, let $g\in \Hom_{R}(I,\ I_{p})$. Since
$I_{p}$ is principal in $R_{p}$, $I_{p}=aR_{p}$ for some $a\in I$.
But $g(a)\in I_{p}$. Hence $g(a)=\frac{ca}{s}$ for some $c\in R$
and $s\in R \setminus p$. Let $x\in I$. Then $\frac{x}{1}\in
I_{p}=aR_{p}$. Hence $\frac{x}{1}=\frac{ra}{u}$ for some $r\in R$
and  $u \in R\setminus p$. So there exists $t\in R\setminus p$
such that $tux=tra$. Now, let $f:I\longrightarrow I$ be the
multiplication by $c$. (i.e., for $x\in I$,  $f(x)=cx$). Then
$f\in \Hom_{R}(I,\ I)$ and we have
$$\phi(\frac{f}{s})(x)=\frac{f(x)}{s}=\frac{cx}{s}=\frac{c}{s}\frac{x}{1}=\frac{cra}{su}=\frac{r}{u}g(a)=\frac{1}{tu}g(tra)=\frac{1}{tu}g(txu)=g(x).$$
Therefore $\phi$ is surjective and hence  an isomorphism, as
desired.

$R$ fqp-ring $\Rightarrow$ $R$ Gaussian

Recall that, if $(R,\m)$ is a local ring with maximal ideal $\m$,
then $R$ is a Gaussian ring if and only if for any two elements
$a$, $b$ in $R$, $(a, b)^{2} = (a^{2})\ \mbox{or}\ (b^{2})$ and if
$(a, b)^{2} = (a^{2})$ and $ab = 0$, then $b^{2} = 0$
{\cite[Theorem 2.2 (d)]{BG2}}.

Let $R$ be an fqp-ring and let $P$ be any prime ideal of $R$. Then
by Lemma \ref{F5} $R_{p}$ is a local fqp-ring. Let $a,\ b\ \in
R_{P}$. We investigate two cases. The first case is $(a,\ b)=(a)$
or $(b)$, say $(b)$. So $(a,\ b)^{2}=(b^{2})$. Now assume that
$ab=0$. Since $a\in (b)$, $a=cb$ for some $c\in R$. Therefore
$a^{2}=cab=0$. The second case is $I:=(a,\ b)$ with $I\neq (a)$
and $I\neq (b)$. Necessarily, $a\neq 0$ and $b\neq 0$.
 By Lemma \ref{F6}, $a^{2}=b^{2}=ab=0$. Both cases satisfy the conditions that were mentioned at the beginning of this proof (The conditions  of
{\cite[Theorem 2.2 (d)]{BG2}}). Hence $R_{p}$ is Gaussian. But $p$
being an arbitrary prime ideal of $R$ and the Gaussian notion
being a local property, then $R$ is Gaussian.

To prove that the implications are irreversible in general, we
will use  the following theorem to build examples for this
purpose.

\begin{thm}[{\cite[Theorem 4.4]{AJK}}] \label{F7}
Let $(A,\ \m)$ be a local ring and $E$ a nonzero
$\frac{A}{\m}$-vector space. Let $R:=A \ltimes E$ be the trivial
ring extension of $A$ by $E$. Then $R$ is an fqp-ring if and only
if $\m^{2}=0$.
\end{thm}

The proof of this theorem depends on the following lemmas.

\begin{lemma} [{\cite[Theorem 2]{SM}}] \label{F8}  Let $R$ be a local fqp-ring which is not a chained
ring. Then $(\Nil(R))^{2}=0$.
\end{lemma}

\begin{lemma} [{\cite[Lemma 4.5]{AJK}}] \label{F9}  Let $R$ be a local fqp-ring which is not a chained
ring. Then $\Ze(R)=\Nil(R)$.
\end{lemma}

\begin{proof}
 We always have $\Nil(R)\subseteq \Ze(R)$. Now, let $s\in \Ze(R)$.
Then there exists $t\neq0 \in R$ such that $st=0$. Since $R$ is
not chained, there exist  nonzero elements $x,\ y\ \in R$ such
that $(x)$ and $(y)$ are incomparable. By Lemma \ref{F6},
$x^{2}=xy=y^{2}=0$. Either $(x)$ and $(s)$ are incomparable and
hence, by Lemma \ref{F6}, $s^{2}=0$. Whence $s\in \Nil(R)$. Or
$(x)$ and $(s)$ are comparable. In this case, either $s=rx$ for
some $r\in R$ which implies that $s^{2}=r^{2}x^{2}=0$ and hence
$s\in \Nil(R)$. Or $x=sx^{\prime}$ for some $x^{\prime} \in R$.
Same arguments applied to $(s)$ and $(y)$ yield either $s\in
\Nil(R)$ or $y=sy^{\prime}$ for some $y^{\prime} \in R$. Since
$(x)$ and $(y)$ are incomparable, $(x^{\prime})$ and
$(y^{\prime})$ are incomparable. Hence, by Lemma \ref{F6},
$(x^{\prime})\cap(y^{\prime})=0$. If $(x^{\prime})$ and $(t)$ are
incomparable, then by Lemma \ref{F6}, $\Ann(x^{\prime})=\Ann(t)$.
So that $s\in \Ann(x^{\prime})$ which implies that
$x=sx^{\prime}=0$, absurd. If $(t)\subseteq (x^{\prime})$, then
$(t)\cap(y^{\prime})\ \subseteq (x^{\prime})\cap(y^{\prime})=0$.
So $(t)$ and $(y^{\prime})$ are incomparable, whence similar
arguments as above yield $y=0$, absurd. Last, if $(x^{\prime})
\subseteq (t)$, then $x^{\prime}=r^{\prime}t$ for some $r^{\prime}
\in R$. Hence $x=sx^{\prime}=str^{\prime}=0$, absurd. Therefore
all the possible cases lead to  $s\in \Nil(R)$. Consequently,
$\Ze(R)=\Nil(R)$.\end{proof}
\begin{lemma}[{\cite[Lemma 4.6]{AJK}}] \label{F10}
Let $(R,\ \m)$ be a local ring such that  $\m^{2}=0$. Then $R$ is
an fqp-ring.
\end{lemma}

\begin{proof}
Let $I$ be a nonzero proper finitely generated ideal of $R$. Then
$I\subseteq \m$ and $\m I=0$. Hence $\m \subseteq \Ann(I)$, whence
$\m =\Ann(I)$ $(I\neq 0)$. So that $\frac{R}{\Ann(I)}\cong
\frac{A}{\m}$ which implies that $I$ is a free
$\frac{R}{\Ann(I)}$-module, hence projective over
$\frac{R}{\Ann(I)}$. By Lemma \ref{F4}, $I$ is quasi-projective.
Consequently, $R$ is an fqp-ring.
\end{proof}

\textbf{Proof of Theorem \ref{F7}.}
Assume that $R$ is an fqp-ring. We may suppose that $A$ is not a
field. Then $R$ is not a chained ring since $((a,\ 0)$ and $((0,\
e))$ are incomparable where $a\neq 0\ \in \m$ and $e=(1,\ 0,\ 0,\
...)\ \in E$. Also $R$ is local with maximal $\m \ltimes E$. By
Lemma \ref{F9}, $\Ze(R)=\Nil(R)$. But $\m \ltimes E=\Ze(R)$. For,
let $(a,e)\ \in \m \ltimes E$. Since $E$ is an
$\frac{A}{\m}$-vector space, $(a,e)(0,e)=(0,ae)=(0,0)$. Hence $\m
\ltimes E\subseteq \Ze(R)$. The other inclusion holds since
$\Ze(R)$ is an ideal. Hence $\m \ltimes E=\Nil(R)$. By Lemma
\ref{F8}, $(\Nil(R))^{2}=0= (\m \ltimes E)^{2}$. Consequently,
$\m^{2}=0$.

Conversely, $\m^{2}=0$ implies $(\m \ltimes E)^{2}=0$ and hence by
Lemma \ref{F10}, $R$ is an fqp-ring. $\Box$
\bigskip

Now we can use Theorem \ref{F7} to construct examples which prove
that the implications in Theorem \ref{F3} cannot be reversed in
general. The following is an example of an fqp-ring which is not
an arithmetical ring
\begin{example} \rm \label{F11}

$R:=\frac{\mathbb{R}[X]}{(X^{2})} \ltimes \mathbb{R}$ is an
fqp-ring by Theorem \ref{F7}, since $R$ is   local  with a
nilpotent maximal ideal $\frac{(X)}{(X^{2})} \ltimes \mathbb{R}$.
Also, since $\frac{\mathbb{R}[X]}{(X^{2})}$ is not a field, $R$ is
not arithmetical by Theorem \ref{T3}.
\end{example}

The following is an example of a Gaussian ring which is not an
fqp-ring.

\begin{example} \rm \label{F12}

$R:=\mathbb{R}[X]_{(X)} \ltimes \mathbb{R}$  is Gaussian by
Theorem \ref{T3}. Also, by Theorem \ref{F7}, $R$ is not an
fqp-ring.
\end{example}

Now the natural question is what are the values of the weak global
dimension of an arbitrary fqp-ring$?$  The answer is given by the
following theorem.
\begin{thm} [{\cite[Theorem 3.11]{AJK}}] \label{F13}
Let $R$ be an fqp-ring. Then  $\wgdim(R)=$ 0, 1, or $\infty$.
\end{thm}

\begin{proof}
Since $\wgdim(R)=\sup\{\wgdim(R_{p})\mid p\ \mbox{prime ideal
of}\ R\}$, one can assume that $R$ is a local fqp-ring. If $R$ is reduced, then
$\wgdim(R)\leq 1$ by Lemma \ref{B5}. If $R$ is not reduced, then $\Nil(R)\neq 0$. By
Lemma \ref{F8}, either $(\Nil(R))^{2}= 0$, in this case,
$\wgdim(R)=\infty$ by Theorem \ref{B7} (since an fqp-ring is
Gaussian); or $R$ is a chained ring with zero divisors
($\Nil(R)\neq 0$), in this case $\wgdim(R)=\infty$ by Theorem
\ref{A3}. Consequently,   $\wgdim(R)=$ 0, 1, or $\infty$.
\end{proof}

It is clear that Theorem \ref{F13} generalizes Osofsky's Theorem
on the weak global dimension of arithmetical rings (Theorem
\ref{A3}) and partially resolves  Bazzoni-Glaz Conjecture on
Gaussian rings (Conjecture 4.13).
\end{section}


\end{document}